\documentclass{article}

\usepackage{latexsym}
\usepackage{amsthm}

\title{Cyclicity in families of circle maps}
\author{O. Kozlovski}

\newtheorem{theorem}{Theorem}[section]
\newtheorem{mytheorem}{Theorem}

\newtheorem{proposition}{Proposition}[section]
\newtheorem{lemma}{Lemma}[section]

\theoremstyle{definition}
\newtheorem{definition}{Definition}

\newenvironment{theorem-proof}[1]
{\underline{Proof of Theorem~\ref{#1}} \newline}
{\hfill {$\Box$}}

\newfont{\Bb}{msbm10}
\newcommand{\N}{\mbox{\Bb N}}
\newcommand{\Z}{\mbox{\Bb Z}}
\newcommand{\R}{\mbox{\Bb R}}
\newcommand{\C}{\mbox{\Bb C}}

\newcommand{\ic}{\mbox{\cal IC}}

\newcommand{\supp}{\mathop{\mbox{supp}}}

\newcommand{\aparam}{a}

\begin{document}

\maketitle

\begin{abstract}
  In this paper we will study families of circle maps of the form
  $x\mapsto x+2\pi r + \aparam f(x) \pmod{2\pi}$ and investigate how many periodic
  trajectories maps from this family can have for a ``typical''
  function $f$ provided the parameter $\aparam$ is small.
\end{abstract}

\section{Introduction}
\label{sec:introduction}

The Hilbert-Arnold problem (\cite{arnold1985}, see also discussion in
\cite{ilyashenko2002}) asks if a generic family of planar vector
fields has a uniformly bounded number of isolated limit cycles.
Originally this problem was formulated just for planar vector fields
(or vector fields on 2 dimensional sphere), however the same question
can be posed for vector fields on other manifolds or even for families
of a maps from some manifold to itself.  In this paper we consider a
case of specific families of diffeomorphisms of the circle.

Consider a family of circle diffeomorphisms of this form:
$$
F_{r,\aparam} : x \mapsto x+2\pi r + \aparam f(x) \pmod{2\pi},
$$
where $f$ is some periodic function, $r \in \R$, $\aparam \in
[-\aparam_0,\aparam_0]$. For a given $2\pi$-periodic function $f$
this family will be called the \emph{corresponding family}.

A number of people studied such families for different choices of the
function $f$. If $f=\sin$, the family $F$ is usually called Arnold's
family.

For these families we will study periodic trajectories which originate
from periodic points of the rigid rotation and
investigate whether for a ``typical'' family $F$ there is a bound on
the number of periodic trajectories which are born when
$\aparam$ is small. In the original Hilbert-Arnold problem ``typical'' means
Baire generic. Here we will consider several other notions of
typicality and we will see that the answer might depend on what notion
of typicality is used.

We will see that in many cases a ``typical'' family will have infinite
cyclicity: this means that for any $N\in \N$ there exist parameter
values $r$ and $a$ such that the map $F_{r,a}$ has more than $N$
attracting periodic trajectories. On the other hand, families which
have finite cyclicity certainly exist: if $f$ is a trigonometric
polynomial of degree $d$, then the number of attracting periodic
trajectories of the map $F_{r,a}$ is bounded by $d$, see \cite{yakobson1985}. 

All the questions we pose here for the specific family $F_{r,\aparam}$
can be asked for general families of circle maps. Interestingly enough
a generic non trivial family of $C^k$ circle diffeomorphisms has
infinite cyclicity. This follows from Herman's theorem though to the
best of my knowledge it cannot be found in the literature. It seems
that such a statement holds only for families of diffeomorphisms. I
conjecture that a generic family of critical circle maps (i.e. maps
which have points where the derivative of the map is zero) has finite
cyclicity. The critical points in such families do not create
significant problems for the proof of this conjecture because periodic
attractors near critical points can be controlled by the negative
Schwarzian derivative condition (see \cite{kozlovski2000}). On the
other hand, periodic attractors of high period born in a perturbation
of a neutral fixed point are hard to analyse.

{\it Acknowledgements.\ } The author would like to thank D. Preiss for a
number of very useful discussions.

\section{Statements of results}
\label{sec:null-sets-1}

For a diffeomorphism $g$ of a circle let $N(g)$ denote the
number of attracting periodic trajectories of the map
$g$ and let $\rho(g)$ denote the rotation number of $g$. For the
family $F_{r,\aparam}: x \mapsto x+2\pi r + \aparam f(x) \pmod{2\pi}$ of circle maps define
$$
B(F,\rho_0)=\limsup_{\aparam_1 \to 0+}
\sup_{\rho(F_{r,\aparam})=\rho_0,\, |\aparam|<\aparam_1}
N(F_{r,\aparam}).
$$
So, $B$ quantifies how many periodic attractors of a given rotation
number are born when we perturb the rigid rotation.

\begin{definition}
  The family of circle maps $F_{r,\aparam} : x \mapsto x+2\pi
  r+\aparam f(x) \pmod{2\pi}$ is said to have finite infinitesimal cyclicity if
$B(F, \rho)$ is uniformly bounded for all $\rho$. Otherwise, $F$ has
infinite infinitesimal cyclicity.
\end{definition}

We will study cyclicity of families $F$ in different functional
spaces and use the following notation. Let $X$ be some space of
real functions on the circle. $\ic_N(X)\subset X$ will denote a set of
functions $f$ in $X$ whose corresponding families  $F_{r,\aparam} : x \mapsto x+2\pi r+\aparam
f(x) \pmod{2\pi}$ 
have infinitesimal cyclicity bounded by $N$, i.e. $B(F,\rho)\le
N$ for all $\rho \in \R$.
 $\ic_{\mbox{{\small fin}}}(X)=\cup_{N=1}^\infty \ic_N(X)$
will denote functions whose corresponding families have finite
infinitesimal cyclicity and $\ic_\infty(X)=X\setminus \ic_{\mbox{{\small fin}}}(X)$
-- functions whose corresponding families have infinite cyclicity.

\subsection{Topologically small sets}
\label{sec:topol-small-sets}

A standard way to define a topologically small set is to set a
subset $E$ of a topological space $X$ to be small if $E$ is a
countable union of nowhere dense subsets of $X$. Such a subset $E$ is
also called \emph{meagre}. The complement of a meagre set is called a
\emph{residual} set and if some property holds for all elements of a
residual set, we say that this property is Baire generic.

Our first theorem states that the infinite cyclicity is a Baire
generic property:
\begin{mytheorem} \label{thr:iic}
  The set $\ic_\infty(C^k(S))$,
  $k=2,\ldots$, as well as the set
  $\ic_\infty(C^\omega_\delta(S))$, is residual.

  In other words, for a Baire generic function $f \in C^k(S)$,
  $k=2,\ldots,\omega$, the corresponding family $F$ has
  infinite infinitesimal cyclicity.
\end{mytheorem}

Here $C^\omega_\delta (S)$ denotes the space of real analytic $2\pi$
periodic functions which can be analytically continued to the strip
$\R \times (-\delta i, \delta i) \subset \C$.  The norm in this space
is given by
$$
\|f\|_{C^\omega_\delta(S)}=\sup_{Im(z)<\delta} |f(z)|.
$$

\subsection{$\Gamma_b^l$ prevalence}
\label{sec:metric-prevalence}

It is well known that  large in topological sense sets can have
zero Lebesgue measure and be invisible from ``probabilistic'' point of
view. The most known example here is the set of Liouville numbers: this
set is residual but has Lebesgue measure zero.

The space of circle functions is infinite dimensional and there is no
natural definition of the Lebesgue measure in infinite dimensional spaces.
Thus, there is no natural (or unique) way to define zero Lebesgue
measure sets. Several different approaches to define metric prevalent
sets have been suggested, we will study a few of them and investigate their
relations to the cyclicity of families of circle maps.

First we consider a definition of a metric prevalent set which employs
a mixture of the topological and finite dimensional metric prevalence.

\begin{definition}
  Let $X$ be some vector space. Let $I^b$ denote a unit cube in
  $\R^b$.  A subset $P\subset X$ is called
  \emph{$\Gamma^l_b$-prevalent} if for a generic $C^l$ family $g: I^b
  \to X$ the set $g^{-1}(P)$ has full Lebesgue measure.
\end{definition}

Though this notion of the prevalence is arguably most popular in
dynamics, $\Gamma_b^l$ prevalent sets cannot be considered as a true
generalisation of full Lebesgue measure sets: one can construct a zero
Lebesgue measure set in $\R^2$ which is $\Gamma^l_1$-prevalent, see
\cite{preiss2012} for the details.

In our setting the space $X$ is a space of circle functions, so to make
use of the $\Gamma^l_b$ prevalence we will consider families with values
in $X$ as in the paragraph above. On the other hand, for a map in
$X$ we will also consider the corresponding family of circle maps. Notice
that we use the same word ``family'' in these two different settings,
however we hope this will not cause any confusion.

\begin{mytheorem} \label{thr:ck-ic}
 For a Baire generic $C^l$   family
$f:I^b \to C^k(S)$ for every parameter $t\in I^b$ the infinitesimal
cyclicity of the corresponding family $x\mapsto x+2\pi r+\aparam
f_t(x)$ is infinite. Here $b \ge 1$, $l \ge 0$, $k=2,\ldots,
\omega$\footnote{if $k=\omega$, then we mean the space
  $C^\omega_\delta(S)$ for some $\delta >0$.}.
\end{mytheorem}

Notice that the claim of this theorem is stronger than just
the $\Gamma^l_b$ prevalence: the corresponding family has infinite infinitesimal
cyclicity for \emph{every} value of the parameter.

\subsection{Haar null sets}
\label{sec:null-sets}

Another way do describe ``small'' sets in infinite dimensional spaces
was suggested in \cite{christensen1972} and \cite{hunt1992}. Let $X$
be a complete separable normed linear space.

\begin{definition}
  A Borel subset $E$ of $X$ is called a \emph{Haar null set} if there
  is a Borel probability measure $\mu$ such that
  $\mu(x+E)=0$ for all $x\in X$.
\end{definition}

In \cite{hunt1992} such sets are called \emph{shy}. The complement of a
Haar null set will be called \emph{Haar prevalent} to avoid confusion
with prevalence in the $\Gamma^l_b$ sense.

Haar null sets have many properties which zero Lebesgue measure sets
enjoy in finite dimensional spaces, e.g. the countable union of Haar
null sets is a Haar null set as well. See \cite{hunt1992},
\cite{benyamini2000}, \cite{ott2005} for more
information.

The next theorem shows that in the finite smoothness case the infinite cyclicity
is Haar prevalent.

\begin{mytheorem} \label{thr:cknull}
  The set $\ic_\infty(C^k(S))$ of functions $f$
  such that the corresponding families $F_{r,\aparam} :x \mapsto
  x+2\pi r+\aparam f(x) \pmod{2\pi}$ 
  have infinite infinitesimal cyclicity is Haar prevalent.
  Here $k=2,3,\ldots$.
\end{mytheorem}

Surprisingly enough in the analytic case the situations is completely
opposite: the infinite infinitesimal cyclicity is Haar null. Moreover,
a Haar typical family has only one attracting cycle if its period high
enough. To make this precise we need the following definition:

\begin{definition}
  We say that the infinitesimal cyclicity of the family
  $F_{r,\aparam}$ is essentially bounded by $N$ if there
  exists $Q\in
  \N$ such that for any $p,q\in \N, q\ge Q$ we have $B(F,\frac pq)\le
  N$.
\end{definition}
In other words, if the period is large enough, at most $N$ periodic
attractors of this period can appear from the rigid rotation.
It is easy to show that if the infinitesimal cyclicity of a family is essentially
bounded by some $N$, then it is finite. For a given space $X$ of
circle functions $\ic_N^e(X)$ will denote the set of functions $f$ whose
corresponding families $F_{r,\aparam} :x \mapsto
  x+2\pi r+\aparam f(x) \pmod{2\pi}$ have infinitesimal cyclicity essentially
bounded by $N$.

\begin{mytheorem} \label{thr:comeganull}
  The set $\ic^e_1(C^\omega_\delta(S))$ is Haar prevalent.
\end{mytheorem}

The notion of Haar null set can be strengthened and next we define
\emph{cube null} sets. From the definition it will be clear that any
cube null set is also Haar null.

Consider the Hilbert cube $[0,1]^{\aleph_0}$ and a linear map
$T:[0,1]^{\aleph_0} \to X$ defined by
$T((t_1,t_2,\ldots))=x_0+\sum_{n=1}^\infty t_n x_n$ where $x_0,x_1,\ldots
\in X$ such that the vectors $x_1,x_2,\ldots$ are linearly
independent, have a dense linear span in $X$ and $\sum_{n=1}^\infty \|x_n\| <
\infty$. The image of the standard product measure on
$[0,1]^{\aleph_0}$ under $T$ is called a \emph{cube} measure and
denoted by $\mu_T$. By definition, we call a set $E$ \emph{cube null}
if $\mu_T(E)=0$ for any cube measure. The complement of a cube null
set we will call cube prevalent.

Similarly one can define \emph{Gauss} null sets where the Gauss
distributions are used instead of the uniform distributions. The Gauss
null sets defined in this way appear to be cube null and vice verse.

The next statement shows that the previous theorems cannot be
strengthened to the case of the cube prevalence.

\begin{mytheorem}\label{thr:cubenull}
  \hspace{1mm}
  \begin{itemize}
  \item The set $\ic_\infty(C^k(S))$ is NOT cube prevalent.
  \item The set $\ic^e_1(C^\omega_\delta(S))$ as well as $\ic_{\mbox{{\small fin}}}(C^\omega_\delta(S))$
    are NOT cube prevalent.
  \end{itemize}
\end{mytheorem}

Notice that since the sets $\ic_\infty(C^k(S))$ and $\ic_{\mbox{{\small fin}}}(C^\omega_\delta(S))$ are Haar
prevalent they cannot be cube null either.

\subsection{$\sigma$-porous sets}
\label{sec:sigma-porous-sets}

Yet another approach to define typical sets in infinite dimensional
spaces is to use a notion of $\sigma$-porous sets. 

Let $X$ be a normed linear space. $B(x,r)=\{y\in X:\, \|y-x\|<r\}$ will
denote a ball of radius $r$ centred at $x\in X$.

\begin{definition}
  A subset $E$ of $X$ is called \emph{porous at $x$} if there
  is $c>0$ such that for any $\epsilon > 0$ there is $y\in X$ such
  that $\|y - x\| < \epsilon$ and  $B(y,c\|y-x\|) \cap E=\emptyset$.

  A subset $E \subset X$ is called \emph{porous} if $E$ is porous at every $x\in E$.

  A subset is called \emph{$\sigma$-porous} if it is a union of countably
  many porous sets.
\end{definition}

It is clear from the definition that porous sets are nowhere
dense and that the complement of a $\sigma$-porous set is residual.
The Lebesgue density theorem implies that if $X$ is finite
dimensional, then every $\sigma$-porous set has zero Lebesgue measure.

There is a direct link between complements of $\sigma$-porous sets and
$\Gamma^1_1$-prevalent sets discussed in one of the previous sections.

\begin{theorem}[\cite{preiss2012}]\label{thr:nsbs}
  Let $X$ be a Banach space with a separable dual and $E\subset X$ be $\sigma$-porous.
  Then for a Baire generic $C^1$ family $f:[a,b] \to X$ the set
  $f^{-1}(E)$ has zero Lebesgue measure.
\end{theorem}

We will show that the $C^k$ functions with finite cyclicity form a
$\sigma$-porous set if $k$ is finite, however Theorem~\ref{thr:nsbs}
would not imply Theorem~\ref{thr:ck-ic} because the dual to the space
$C^k(S)$ is not separable. Considering Sobolev spaces instead of $C^k(S)$
spaces would make application of Theorem~\ref{thr:nsbs} possible (and
all theorems we prove here can be easily generalised to the Sobolev
spaces), however we still would get a statement much weaker than
that of Theorem~\ref{thr:ck-ic}.

Next theorem tells us that in the case of smooth (non analytic)
functions the infinite cyclicity prevails once again.

\begin{mytheorem}\label{thr:ckporous}
  The set $\ic_{\mbox{{\small fin}}}(C^k(S))$, $k=2,\ldots$, of functions $f$
  whose corresponding families $F_{r,\aparam} :x \mapsto
  x+2\pi r+\aparam f(x) \pmod{2\pi}$ have finite infinitesimal cyclicity is
  $\sigma$-porous.
\end{mytheorem}

It is not clear if the sets $\ic_{\mbox{{\small
      fin}}}(C^\omega_\delta(S))$ or $\ic_\infty(C^\omega_\delta(S))$
are $\sigma$-porous or not.

\subsection{Summary}
\label{sec:summary}

For the convenience of the reader all results formulated in the
previous sections are summarised in the following table:

\vspace{3mm}

\begin{tabular}{|l|c|c|}
\hline
Prevalence & \hspace{7mm} $C^k(S)$ \hspace{7mm} & \hspace{7mm}
$C^\omega_\delta(S)$ \hspace{7mm}  \\
\hline
Baire residual & $\ic_\infty$ &  $\ic_\infty$\\
$\Gamma^l_b$ prevalent &  $\ic_\infty$ &  $\ic_\infty$ \\
Haar prevalent &  $\ic_\infty$ &  $\ic_1^e$\\
Cube prevalent & none & none \\
Complement to $\sigma$-porous &  $\ic_\infty$ & ? \\
\hline
\end{tabular}

\subsection{Torus vector fields}
\label{sec:torus-vector-fields}

Instead of perturbations of rigid rotation of the circle we can also
study perturbations of constant vectors fields on the torus.

For two functions  $v_1$, $v_2$ on the torus consider a family $V_{\alpha,a}$
of vector fields on the torus given by 
\begin{eqnarray*}
  V_{\alpha,a}(x_1,x_2)=
  \left(
    \begin{array}{c}
      \cos \alpha\\
      \sin \alpha
    \end{array}
  \right)  
 + 
 a \left(
    \begin{array}{c}
      v_1(x_1,x_2)\\
      v_2(x_1,x_2)
    \end{array}
    \right)
\end{eqnarray*}
where $\alpha \in \R$ and $a \in
(-a_0,a_0)$ are parameters. We could also consider a family
$$
  \left(
    \begin{array}{c}
      \cos \alpha\\
      \sin \alpha
    \end{array}
  \right)  
 + 
  \left(
    \begin{array}{c}
      a_1 v_1(x_1,x_2)\\
      a_2 v_2(x_1,x_2)
    \end{array}
    \right)
$$
which depends on two parameters $a_1$ and $a_2$ instead of only one
parameter $a$.

The definition of the infinitesimal cyclicity for such families is
straightforward, one should count the number of attracting periodic
limit cycles. Then the statements of Theorems~\ref{thr:iic}, \ref{thr:ck-ic}, \ref{thr:cknull},
\ref{thr:comeganull}, \ref{thr:cubenull}, \ref{thr:cubenull} and \ref{thr:ckporous} can be
reformulated in the obvious way for such families of vector fields and
the statements remain correct. In
Section~\ref{sec:tourus-case} we will justify why this can be done,
but we will not reprove all these theorems as their proofs are almost
identical to the ones corresponding to the families of circle maps.

\section{Perturbations of rigid rotation}
\label{sec:pert-rigid-rotat}

Consider a family $F_{r,\aparam} : x \mapsto x+2\pi r+\aparam f(x)
\pmod{2\pi}$. The $n$-th iterate of $F_{r,\aparam}$ is
$$
F_{r,\aparam}^n(x)=x+2\pi rn + \aparam f(x) + \aparam f(x+2\pi
r+\aparam f(x))+\cdots.
$$
This formula shows that if $F_{r,\aparam}$ has a periodic trajectory
of period $q$, then there is $p\in \Z$ such that
$$
|r-\frac pq| \le \frac \aparam{2\pi} \|f\|_{C^0(S)}.
$$

If $f\in C^2(S)$, and $\aparam$ is small, then the above formula implies
$$
F_{r,\aparam}^n(x)=x+2\pi rn+\aparam \sum_{k=0}^{n-1} f(x+kr) +
O(\aparam^2).
$$

Suppose that function $f$ can be represented as the following Fourier series:
$$
f(x)=\sum_{l=-\infty}^{+\infty} {\tilde f}_l e^{ilx}.
$$
A straightforward computation shows that if $r=\frac pq$, where
$p,q\in \N$, then
$$
\sum_{k=0}^{q-1} f(x+kr)=q\sum_{l=-\infty}^{+\infty} {\tilde f}_{ql}e^{iqlx}.
$$

The following proposition immediately follows from the discussion above.
\begin{proposition}
  \label{pr:ck}
  Let $f$ be in $C^2(S)$. Suppose there exists $R\in \R$ such that the equation 
  $$
  \sum_{l=-\infty}^{+\infty} {\tilde f}_{ql}e^{iqlx}=R
  $$
  has $d$ simple roots. Then for all sufficiently small non zero $\aparam $
  there exists $r(\aparam)\in [0,2\pi)$ such that the map
  $F_{r(\aparam),\aparam}$ has at least $d$ isolated periodic trajectories of
  period $q$.
\end{proposition}

The converse also holds at least in the analytic case.

\begin{proposition}
  \label{pr:comega}
  Let $f\in C^\omega_\delta(S)$.  Suppose for any $R\in \R$ the equation 
  $$
  \sum_{l=-\infty}^{+\infty} {\tilde f}_{ql}e^{iqlx}=R
  $$
  has at most $d$ real roots (counted with multiplicity). Then there
  exists $\aparam_0>0$ such that for any $r\in \R$ and $\aparam \in
  (-\aparam_0,\aparam_0)\setminus\{0\}$ the map $F_{r,\aparam}$ has
  at most $d$ periodic trajectories of period $q$.
\end{proposition}

The proof of this proposition is a straightforward application of the
argument principle. Indeed, denote $g(x)=\sum_{l=-\infty}^{+\infty}
{\tilde f}_{ql}e^{iqlx}$. For every value $R_0\in \R$ there exists
$\beta(R_0) \in (0,\delta)$ and $A(R_0)>0$ such that the equation $g(x)=R_0$ has at most
$d$ roots in the strip $Im(x)\le \beta(R)$ and $\sup_{y\in \R}|g(\pm i
\beta(R_0) + y) - R_0| > A(R_0)$. 
Moreover, there exists $r(R_0)>0$ such that  $\sup_{y\in \R}|g(\pm i
\beta(R_0) + y) - R| > A(R_0)$ 
for all $R \in
(R_0-r(R_0),R_0+r(R_0))$. By the argument principle for these values
of $R$ the equation $g(x)=R$ still has at most $d$ roots in the strip
$Im(x)\le \beta(R)$. Also, if $|R|>\|g\|_{C^\omega_\delta(S)}$, then
the equation $g(x)=R$ does not have roots at all. Consider a cover of
the segment
$(-2\|g\|_{C^\omega_\delta(S)},2\|g\|_{C^\omega_\delta(S)})$ by
intervals of the form $(R-r(R),R+r(R))$ and take a finite subcover
$(R_1-r(R_1),R_1+r(R_1)),\ldots,(R_L-r(R_L),R_L+r(R_L))$. Set
$\beta=\min(\|g\|_{C^\omega_\delta(S)}), A(R_1),\ldots, A(R_L))$. It
is easy to see that due to the argument principle for all $R\in \R$ the equation
$g(x)+h(x)=R$ has at most $d$ roots if
$\|h\|_{C^\omega_\delta(S)}<\beta$. This proves the proposition.

\section{Sturm-Hurwitz theorem}
\label{sec:sturm-hurw-theor}

In order to provide a lower bound on the number of zeros of a function
we employ a version of the Sturm-Hurwitz theorem.  The original
Sturm-Hurwitz theorem states that if all harmonic components of order
less than $n$ of a real continuous function $g$ vanish, then the
function $g$ has at least $2n$ zeros. We need somewhat sharper
statement where the harmonic components do not vanish but are small.

\begin{theorem}\label{thr:sh}
  Let $\phi$ be a real differentiable $2\pi$ periodic function and $a_n$ be
  its Fourier coefficients, i.e.
  $a_n=\frac 1{2\pi}\int_0^{2\pi} \phi(x) e^{-inx}\,dx$.
  Moreover, assume that for some $n \in \N$
  $$
  \sum_{k=-n+1}^{n-1} |a_k| < 2^{-2n+3} |a_{n}|.
  $$
  Then on the interval $[0,2\pi)$ the function $\phi$ changes its sign at
  least $2n$ times.
\end{theorem}

In the proof of this theorem the following statement will be used:
\begin{proposition}
  Let $\phi$ be a \emph{positive} real differentiable $2\pi$ periodic function and $a_n$ be
  its Fourier coefficients. Then
  $$|a_0| \ge 2|a_n|$$
  for all $n\neq 0$.
\end{proposition}
The proof of this proposition can be found in \cite{Polya1976}(Part IV, N
51, page 71) where it is proved in the case when $\phi$ is a
trigonometric polynomial. Since the Fourier partial sums of a
differentiable function converge uniformly the statement follows.

\begin{theorem-proof}{thr:sh}
  We will prove this theorem by induction following the original proof
  of Hurwitz~\cite{hurwitz1903}.

  The case $n=1$ follows immediately from the above proposition.

  Now let us make an induction step and suppose the theorem holds for some
  $n$. Suppose that 
  $$
  \sum_{k=-n}^{n} |a_k| < 2^{-2n+1} |a_{n+1}|,
  $$
  but $\phi$ changes its sign less than $2n+2$ times. Again, due to the
  previous proposition the function $\phi$ cannot be positive (or
  negative) for all $x$ and it changes its sign at least twice. Denote
  these points where $\phi$ is zero by $x_1$ and $x_2$. Then the function
  $\hat \phi(x)=\phi(x) \sin((x-x_1)/2) \sin((x-x_2)/2)$ is $2\pi$ periodic
  and has less than $2n$ changes of sign.

  Let us compute Fourier coefficients of the function $\hat \phi$. Note that
  $$
  \sin((x-x_1)/2) \sin((x-x_2)/2)= -\frac 14 e^{\frac{x_1+x_2}2 i} e^{-ix}+
  \frac 12 \cos(\frac{x_1-x_2}2)  -\frac 14 e^{-\frac{x_1+x_2}2 i}
  e^{ix}.
  $$
  This implies that Fourier coefficients of $\hat \phi$ are 
  $$
  \hat a_k=
  -\frac 14 e^{-\frac{x_1+x_2}2 i} a_{k-1}
  + \frac 12 \cos(\frac{x_1-x_2}2) a_k
  -\frac 14 e^{\frac{x_1+x_2}2 i} a_{k+1}.
  $$
  Now we can estimate $|\hat a_k|$ in terms of the Fourier
  coefficients of the function $\phi$ :
  \begin{eqnarray*}
    |\hat a_n| & \ge & 
    \frac 14 |a_{n+1}| - \frac 12 |a_n| - \frac 14 |a_{n-1}| \\
    |\hat a_k| & \le & 
    \frac 14 |a_{k-1}| + \frac 12 |a_k| + \frac 14 |a_{k+1}|
  \end{eqnarray*}
  Combining these inequalities together we obtain
  \begin{eqnarray*}
    \sum_{k=-n+1}^{n-1} |\hat a_k| & \le &
    -\frac 34 |a_{-n}| - \frac 14 |a_{-n+1}|
    +\sum_{k=-n}^{n} |\hat a_k|
    - \frac 14 |a_{n-1}| - \frac 34 |a_{n}|\\
    &<&
    2^{-2n+1} |a_{n+1}| 
    - \frac 12 |a_{n-1}| - \frac 32 |a_{n}|\\
    &\le&
    2^{-2n+3} |\hat a_n|.
  \end{eqnarray*}
  We see that we can apply the induction assumption to $\hat \phi$,
  therefore $\hat \phi$ changes its sign at least $2n$ times and
  $\phi$ changes its sign at least $2n+2$ times.
\end{theorem-proof}

\section{Baire genericity of infinite infinitesimal cyclicity}
\label{sec:baire-gener-infin}

\begin{theorem-proof}{thr:iic}
  We will give the prove for the $C^k(S)$ space, for the $C^\omega_\delta(S)$ spaces it works
  the same way.

  Let $p$ be a trigonometric polynomial of degree $N$. Let
  $f(x)=p(x)+c \sin((N+1)dx)$, where $d\in \N$, $c \in \R$ are
  constants, and let $F_{r,\aparam}(x)=x+2\pi r+\aparam
  f(x) \pmod{2\pi}$ be the corresponding family.
  Proposition~\ref{pr:ck} implies that for
  small values of $\aparam$ the map $F_{\frac 1{N+1},\aparam}$ has
  $d$ periodic attractors since
  $$
  F_{\frac 1{N+1},\aparam}^{N+1}(x)=x+\aparam c (N+1) \sin((N+1)d x) +
  O(\aparam^2) \pmod{2\pi}.
  $$
  From the previous theorem we know that there is a neighbourhood of
  $f$ such that for any function in this neighbourhood the
  corresponding family will also have $d$ periodic attracting
  trajectories of period $N+1$.

  The trigonometric polynomials are dense in the space of $C^k$ $2\pi$
  periodic functions, the constant $c$ can be taken arbitrarily small,
  thus we have proved that the set of functions $f$ such that the
  cyclicity of the corresponding family is bounded by $d$ is nowhere
  dense. The union of these sets is meagre and the complement of this
  union is residual.
\end{theorem-proof}

\begin{theorem-proof}{thr:ck-ic}
  This proof uses similar ideas to ones used in the proof of
  Theorem~\ref{thr:iic}. First, consider the space $C^k(S)$ for a
  finite $k$.

  Let $f:I^b \to C^k(S)$ be a $C^l$ family and let ${\tilde f}_n(t)=\frac
  1{2\pi} \int_0^{2\pi} f_t(x) e^{-inx} \,dx$ be its Fourier
  coefficients. The derivative of $f$ with respect to the parameter
  $t$ will be denoted by $f^{(m)}$ where $m=(m_1,\ldots,m_b)$ is a
  multi-index.

  \begin{lemma}\label{lm:uniform}
    For any $\epsilon_0>0$ there exists
    $N=N(\epsilon_0)$ such that for all $t\in I^b$, all $n\in \Z$, $|n|\ge N$ one has
    \begin{equation}\label{eq:3}
      |{\tilde f}_n^{(m)}(t)| \le \epsilon_0 |n|^{-k},
    \end{equation} 
    for all $|m|\le l$.
  \end{lemma}

  \begin{proof}
    Fix $t_0\in I^b$. The  Fourier
    coefficients of the $k$th derivative of $f_{t_0}$ with respect to $x$ have form $n^k \tilde f_n(t_0)$. 
    From the Riemann-Lebesgue lemma we know that $|n^k{\tilde f}_n(t_0)| \to 0$ as
    $n\to \pm \infty$. Hence there exists $N(\epsilon_0,t_0)$ such
    that for all $|n|>N(\epsilon_0,t_0)$ one has $|{\tilde f}_n(t_0)|
    \le \epsilon_0 |n|^{-k}/2$.

    Take $\delta(\epsilon_0,t_0)>0$ so small that
    $\|f_t-f_{t_0}\|_{C^k(S)}<\epsilon_0/2$ if $|t-t_0|<\delta$. Then,
    $|n|^k|\tilde f_n(t)-\tilde f_n(t_0)| < \epsilon_0/2$ for all
    $n\in \Z$. Combining this inequality with the inequality in the
    previous paragraph we get that for $t\in I^b$,
    $|t-t_0|<\delta(\epsilon_0,t_0)$ and $|n|> N(\epsilon_0,t_0)$ the
    inequality~(\ref{eq:3}) holds for $m=0$.

    The set $I^b$ is covered by open balls
    $B(t,\delta(\epsilon_0,t))$. Using compactness of $I^b$ we can take
    a finite subcover $B(t_1,\delta(\epsilon_0,t_1)), \ldots,
    B(t_L,\delta(\epsilon_0,t_L))$
    and set $N(\epsilon_0)=\max_{i\le
      L}N(\epsilon_0,t_i)$. Obviously, the inequality~(\ref{eq:3})
    holds for this choice of $N$ if $m=0$. For the other values of $m$
    the argument is the same.
  \end{proof}

  Now fix some positive integer $d$ and $\epsilon>0$. Consider the family
  $$
  \hat f_t(x) =
  f_t(x) - \sum_{m=1}^{d+1} ({\tilde f}_{-mN}(t) e^{-imNx}+{\tilde f}_{mN}(t)
  e^{imNx}) + \epsilon N^{-k} (d+1)^{-k-1} \sin(N(d+1)x),
  $$
  where $N$ is given by the claim above for $\epsilon_0 = \epsilon
  (d+1)^{-1}/2$.
  The norm of the difference of $f_t$ and $\hat f_t$
  can be estimated as
  $$
  \|\hat f_t -f_t\|_{C^{l,k}(S)} < 2\epsilon.
  $$

  Arguing as before, due to Proposition~\ref{pr:ck} and
 Theorem~\ref{thr:sh} for arbitrary fixed
 $t  \in I^b$ the corresponding family
 $\hat F_{t,r,a}:x \mapsto x+2\pi r + a \hat f_t(x) \pmod{2\pi}$
 has $d+1$ periodic attractors of period $N$ when
  $r=1/N$, and $a$ is small.  Moreover, there exists a neighbourhood of
  the family $\hat f_t$ in $C^{l,k}(S)$ such that the corresponding
  family has the same property.

  Thus, once again we have shown that the set of families $f_t$ whose
  infinitesimal cyclicity is bounded by $d$ for just one value of
  parameter $t$ is nowhere dense. The claim of the theorem follows.

  The case of analytic functions can be done in  exactly the same
  way, just instead of inequality (\ref{eq:3}) one should use
  $$
  |{\tilde f}_n^{(m)}(t)| \le \epsilon_0 e^{-\delta |n|}.
  $$
\end{theorem-proof}

\section{Case of finitely differentiable functions}
\label{sec:case-finit-diff}

\begin{theorem-proof}{thr:ckporous}
  Recall that $\ic_N(C^k(S)) \subset C^k(S)$ denotes a set of functions such that the
  cyclicity of the corresponding families is bounded by $N$. Let $f$ be
  in $\ic_{N-1}(C^k(S))$, ${\tilde f}_n$ denote its Fourier coefficients and fix $\delta >0$.

  Due to the Riemann-Lebesgue lemma we know that $\lim_{n\to \pm
    \infty} n^k|{\tilde f}_n|=0$.
  Hence, there is $d>0$ such that $\sum_{n=1}^{N} (dn)^k |{\tilde f}_{dn}| < \delta/8$.
  Define a new function $\hat f \in C^k(S)$ by  
  $$
  \hat f(x)= f(x) - \sum_{n=1}^{N} ({\tilde f}_{-dn} e^{-idnx} + {\tilde f}_{dn} e^{idnx})+
  \frac{\delta}{2(dN)^k} \sin(dNx).
  $$
  It is easy to see that $\|f-\hat f\|_{C^k(S)} <\delta$.
  
  Consider an arbitrary function $g \in C^k(S)$ such that 
  $
  \|g-\hat f\|_{C^k(S)} < c \delta,
  $
  where
  \begin{eqnarray}
    \label{eq:c}
    c=\frac 1{8 N^{k+1}4^N}
  \end{eqnarray}
  and let ${\tilde g}_n$ denote the Fourier coefficients of $g$. From the
  inequality above it follows that
  $$
  |{\tilde g}_n| <  {c \delta}d^{-k}
  $$
  for $n=d,2d,\ldots, (d-1)N$ and
  $$
  |{\tilde g}_{dN}|> 
  \left(
    \frac{1}{4(dN)^k}-\frac{c}{(dN)^k}
  \right)
  \delta
  >\frac{\delta}{8(dN)^k}
  $$
  Therefore
  \begin{eqnarray*}
    \sum_{n=-N+1}^{N-1} |{\tilde g}_{dn}|
    &<&
    |{\tilde g}_0|+\frac{2c(N-1)\delta}{d^k}\\
    &<&
    |{\tilde g}_0|+\frac{\delta}{4(dN)^k 4^N}\\
    &<&
    |{\tilde g}_0|+ 2^{-2N+1} |{\tilde g}_{dN}|.
  \end{eqnarray*}
  
  Let $g_d(x)$ denote the function $\sum_{l=0}^{d-1} g(x+\frac{2\pi}d
  l)$. The Fourier series for this function is
  $$
  g_d(x)=d\sum_{n=-\infty}^{+\infty} {\tilde g}_{dn} e^{idnx}.
  $$
  Having the estimate on $\sum_{n=-N+1}^{N-1} |{\tilde g}_{dn}|$ we cannot
  apply Theorem~\ref{thr:sh} directly to the function $g_d$ as it
  would require the coefficient $2^{-2dN+3}$
  in front of $|{\tilde g}_{dN}|$ and we have only $2^{-2N+1}$. However, $g_d$ is a
  $\frac{2\pi}d$-periodic function and we can apply
  Proposition~\ref{pr:ck} and
  Theorem~\ref{thr:sh} together with the estimate above to the
  restriction $g_d|_{[0,\frac{2\pi}d]}$ and get that the equation $
  g_d(x)=d{\tilde g}_0 $ has at least $2N$ zeros on the interval
  $[0,\frac{2\pi}d)$ and, therefore, at least $2dN$ zeros on $[0,
  2\pi)$.

  Thus, the family $G_{r,\aparam}:x \mapsto x+ 2\pi r + \aparam g(x)
  \pmod{2\pi}$
  has cyclicity at least $N$ because for small values of $\aparam$ the map
  $G_{\frac 1d -\frac{{\tilde g}_0}{2\pi}\aparam, \aparam}$ 
  has at least $N$ periodic attracting trajectories of period $d$.
  
  We have proved that arbitrarily near any function $f\in C^k(S)$ we
  can find a function $\hat f$, so that the ball
  $B(\hat f, c\|f-\hat f\|)$,
  where $c$ is given by (\ref{eq:c}), does not contain elements
  from the set $\ic_{N-1}(C^k(S))$. Hence, the set $\ic_{N-1}(C^k(S))$ is porous, and the
  union $\cup_{N=1}^\infty \ic_N(C^k(S))$ is $\sigma$-porous.
\end{theorem-proof}

\begin{theorem-proof}{thr:cknull}
  Fix $N\in \N$, we will construct a measure on $C^k(S)$ and show
  that the set $\ic_{N-1}(C^k(S))$ is Haar null
  with respect to this measure.

  Let the sequence $t_m$ be given by 
  $$
  t_m=\frac 1{\sqrt{m} N^{2km}},
  $$
  for $m=1,2,\ldots$. Consider a sequence of independent random
  variables $\phi_m$ on $[0,1]$ such that
  \begin{eqnarray*}
  |\{w\in[0,1]:\phi_m(w)=0\}|   &=& 1 - \frac 1m,\\
  |\{w\in[0,1]:\phi_m(w)=t_m\}| &=& \frac 1m,
  \end{eqnarray*}
  where $|\cdot|$ denotes the Lebesgue measure on $[0,1]$.

  For $w\in [0,1]$ consider a function $\beta_w(x)=\sum_{m=1}^\infty
  \phi_m(w) N^{2km} \sin(N^{2m} x)$. Due to Markov's inequality we
  know that
  $$
  |\{w\in [0,1]: \sum_{m=1}^\infty \phi_m(w) N^{2km} > X\}| \le \frac 1X
  \sum_{m=1}^\infty E(\phi_m) N^{2km},
  $$
  where $X\in \R$ and $E$ denotes the expectation of a random variable. 
  The sum $\sum_{m=1}^\infty E(\phi_m) N^{2km}=\sum_{m=1}^\infty m^{-3/2}$ is convergent,
  therefore the sum $\sum_{m=1}^\infty \phi_m(w) N^{2km}$ converges
  for almost every $w\in [0,1]$. This implies that the function given
  by the Fourier series 
  $$
  \sum_{m=1}^\infty  \phi_m(w) N^{2km} \sin(N^{2m} x)
  $$
  is continuous almost surely. Integrating this function $k$ times
  we see that for almost every $w\in [0,1]$ the function 
  $$
  g_w(x)=\sum_{m=1}^\infty  \phi_m(w) \sin(N^{2m} x)
  $$
  is in $C^k(S)$. The push forward of the Lebesgue measure by the map
  $[0,1] \to C^k(S)$ given by $w \mapsto g_w$ defines a measure on
  $C^k(S)$ which we will denote by $\mu$.

  Take $f\in C^k(S)$ and let ${\tilde f}_n$ be its Fourier coefficients. 
  Our goal is to show that $\mu(f+\ic_{N-1}(C^k(S)))=0$.
  Let $g\in f+\ic_{N-1}(C^k(S))$ and $\tilde g_n$ be the Fourier coefficients
  of $g$. 

  Since $g-f \in \ic_{N-1}(C^k(S))$ and due to Proposition~\ref{pr:ck}
  and Theorem~\ref{thr:sh} we
  know that for all $g\in f+\ic_{N-1}(C^k(S))$
  $$
  \sum_{d=1}^{N-1} |\tilde g_{N^{2m-1}d}-{\tilde f}_{N^{2m-1}d}| \ge 2^{-2N+2} |\tilde g_{N^{2m}}-{\tilde f}_{N^{2m}}|,
  $$
  for all $m\in \N$. If this inequality did not hold for some $m$,
  then the map $x\mapsto x+2\pi N^{1-2m}-{\tilde f}_0\aparam+\aparam (g(x)-f(x))$ would
  have $N$ periodic attractors of period $N^{2m-1}$ for sufficiently
  small values of $\aparam$.

  Let us observe that $\mu$ almost surely $\tilde g_n=0$ if $n$ is not
  of the form $\pm N^{2m}$ for some natural $m$, thus the above
  inequality implies that for $\mu$-almost all $g\in f+\ic_{N-1}(C^k(S))$
  $$
  \sum_{d=1}^{N} |{\tilde f}_{N^{2m-1}d}| \ge 2^{-2N+2} |\tilde g_{N^{2m}}|.
  $$

  Let $c_m=2^{2N-2} \sum_{d=1}^{N} |{\tilde f}_{N^{2m-1}d}|$. We just have
  shown that
  $$
  \mu(f+\ic_{N-1}(C^k(S))) \le \mu(\{g\in C^k(S): \forall m\in \N \,\,\,
  |\tilde g_m| \le c_m\}).
  $$
  A direct computation shows that
  \begin{eqnarray}
  \mu(\{g\in C^k(S): \forall m\in \N \,\,\,|\tilde g_m| \le c_m\})
  &=& \prod_{t_m > c_m} (1-\frac 1m) \\
  &\le& \exp(-\sum_{t_m > c_m} \frac 1m).
  \label{eq:mu}
  \end{eqnarray}
  
  The function $f$ is in $C^k(S)$, hence its $k$th derivative is
  continuous and also belongs to $L^2(S)$  . The Fourier coefficients of
  $f^{(k)}$ are $n^k {\tilde f}_n$ and by Perceval's identity we have
  $\sum_{n=-\infty}^{\infty} |n|^{2k} |{\tilde f}_n|^2 < \infty$. Then
  \begin{eqnarray*}
    \sum_{m=1}^\infty N^{4km}|c_m|^2
    &\le&
    4^{2N-2}\sum_{m=1}^\infty N^{4km+1} \sum_{d=1}^N  |{\tilde f}_{N^{2m-1}d}|^2\\
    &\le&
    4^{2N-2}\sum_{d=1}^N N^{2k+1}d^{-2k} \sum_{m=1}^\infty N^{2(2m-1)k}d^{2k}
    |{\tilde f}_{N^{2m-1}d}|^2 < \infty.
  \end{eqnarray*}

  Therefore,
  \begin{eqnarray*}
    \sum_{t_m \le c_m} \frac 1m
    &=& \sum_{t_m \le c_m} t_m^2 N^{4km}\\
    &\le& \sum_{t_m \le c_m} c_m^2 N^{4km} \\
    &\le& \sum_{m=1}^\infty c_m^2 N^{4km} < \infty.
  \end{eqnarray*}
  This implies that the series $\sum_{t_m > c_m} \frac 1m =
  \sum_{m=1}^\infty \frac 1m - \sum_{t_m \le c_m} \frac 1m$ diverges
  and because of inequality~(\ref{eq:mu}) we get
  $\mu(f+\ic_{N-1}(C^k(S)))=0$.

  So, we have proved that $\ic_{N-1}(C^k(S))$ is a Haar null set. A
  countable union of Haar null sets is Haar null (\cite{hunt1992}),
  hence $\ic_{\mbox{{\small fin}}}(C^k(S))$ is Haar null too.
\end{theorem-proof}

\section{Case of analytic functions}
\label{sec:case-analyt-funct}

First, we will prove this simple lemma:

\begin{lemma} \label{lm:cos}
  Let $n\in \N$ and the function $g$ be in $C^\omega_\delta(S)$ with
  $\|g\|_{C^\omega_\delta(S)} < (\cosh(n\delta)-1)/2$. Then the equation
  \begin{equation}
    r+a\cos(nx)+b\sin(nx)+g(x)=0
    \label{eq:1}
  \end{equation}
  has at most $2n$ real roots on $[0,2\pi)$ for any $r \in \R$  and $a,b\in \R$
  with $a^2+b^2=1$.
\end{lemma}

The proof of this lemma is elementary. By shifting $x$ we can assume
$a=1$ and $b=0$. Consider two cases. If $|r| \ge
(\cosh(n\delta)+1)/2$, then the equation (\ref{eq:1}) has no real roots as $g$
is less than $(\cosh(n\delta)-1)/2$ on the real line.

The image of the segment $[i\delta, i\delta + 2\pi)$ under the map
$z\mapsto \cos(nz)$ is an ellipse with semi-major axis $\cosh(n\delta)$
and semi-minor axis $\sinh(n\delta)$. It is easy to check that the
distance between the segment $[-(\cosh(\delta)+1)/2,
(\cosh(\delta)+1)/2]$ and this ellipse is  $(\cosh(n\delta)-1)/2$. 
Due to the argument principle this implies that the function
$r+\cos(nx)+g(x)$, where $r \in (-(\cosh(n\delta)+1)/2,
(\cosh(n\delta)+1)/2)$, has exactly $2n$ zeros in the rectangle with
vertices $i\delta, 2\pi +i\delta, 2\pi-i\delta, -i\delta$. The
conclusion of lemma follows.

\begin{theorem-proof}{thr:comeganull}
  Let $f$ be in $C^\omega_\delta(S)$. Because of the Cauchy theorem the
  Fourier coefficients of $f$ can be written as 
  $$
  {\tilde f}_n
  =\frac 1{2\pi} \int_{-i\delta}^{-i\delta +2\pi} e^{-inx}f(x)\, dx 
  =\frac {e^{-\delta n}}{2\pi} \int_{0}^{2\pi} e^{-inx}f(x-i\delta)\, dx.
  $$
  Thus, the Fourier coefficients decay exponentially: $\|{\tilde f}_n\|\le
  e^{-\delta |n|} \|f\|_{C^\omega_\delta(S)}$. On the other hand, if
  $\delta' > \delta$, $\|b_n\|\le C e^{-\delta' |n|}$ for some $C>0$,
  and $b_{-n}=\bar b_n$, then the function $\sum_{n=-\infty}^{\infty} b_n
  e^{inx}$ is in $C^\omega_\delta(S)$.
  
  Let $D_1\subset \C$ be a unit disk in the complex plain. Denote the
  complex Hilbert cube $D_1^{\aleph_0}$ as $Q$ and let $\tau$ be the
  standard product measure on $Q$. Consider the map $\phi : Q \to
  C^\omega_\delta(S)$ defined by 
  $$
  \phi((w_1,w_2,\ldots))=\sum_{n=1}^\infty e^{-3/2\delta n}(w_n
  e^{inx}+\bar w_n e^{-inx}).
  $$
  Since the Fourier coefficients decay exponentially faster than
  $e^{-\delta n}$, the map is properly defined for every element of
  $Q$. The image of the measure $\tau$ under this map $\phi$ we denote by $\mu$.
  
  Let $E$ be a set of functions in $C^\omega_\delta(S)$ such that the
  corresponding families have essential infinitesimal cyclicity at least 2.
  Now we will show that for every $f\in C^\omega_\delta(S)$ one has
  $\mu(f+E)=0$.
  
  Let $g$ belong to the support of the measure $\mu$, which is
  equivalent to $g\in \mathop{Image}(\phi)$, and $\tilde g_n$ be the Fourier coefficients
  of $g$. Arguing as before and due to Proposition~\ref{pr:comega} and the
  lemma above, for arbitrarily small non zero values of the parameter
  $\aparam$ there is $r\in \R$ such that the map $x \mapsto x+2\pi r + \aparam(g(x)-f(x))$
  can have at least 2 periodic attracting trajectories of period $N$ if
  \begin{eqnarray}
    |\tilde g_N -{\tilde f}_N|
    &\le& \frac 2{\cosh(N\delta/6) -1}\left\|\sum_{k=2}^\infty (\tilde g_{-kN}-{\tilde f}_{-kN})
      e^{-ikNx}+ (\tilde g_{kN}-{\tilde f}_{kN}) e^{ikNx}
    \right\|_{C^\omega_{\delta/6}(S)} \nonumber\\
    &\le&  \frac 4{\cosh(N\delta/6) -1} \sum_{k=2}^\infty (|\tilde g_{kN}|+|{\tilde f}_{kN}|) e^{\delta kN/6} \nonumber\\
    &\le&  \frac 4{\cosh(N\delta/6) -1} \sum_{k=2}^\infty (e^{-4/3 \delta
      kN}+ \|f\|_{C^\omega_\delta(S)} e^{-5/6 \delta kN})\nonumber\\
    &\le&  \frac 4{\cosh(N\delta/6) -1} \left( \frac{e^{-8/3 \delta
          N}}{1-e^{-4/3 \delta N}} + \|f\|_{C^\omega_\delta(S)}
      \frac{e^{-5/3 \delta N}}{1-e^{-5/6 \delta N}} \right) \nonumber\\
    &\le& C e^{-5/3 \delta N}, \label{eq:2}
  \end{eqnarray}
  where the constant $C$ depends on $\delta$ and $f$, but does not
  depend on $N$.
  
  We know that $\tilde g_N$ is uniformly distributed in the disk of
  radius $e^{-3/2\delta N}$. Thus, the probability that inequality
  (\ref{eq:2}) holds is less than $C^2 e^{-\delta N/3}$.
  
  Consider a sequence of random numbers $p_N=p_N(g)$ such that $p_N$ is
  $1$ if for arbitrarily small non zero values of the parameter
  $\aparam$ there is $r\in \R$ such that the map $x \mapsto x+2\pi r + \aparam(g(x)-f(x))$
  can have at least 2 periodic attracting trajectories of period $N$
  and zero otherwise. 
  
  According to Markov's inequality we have
  $$
  \mu(\{g:\sum_{N=1}^\infty p_N(g)  > X\}) < \frac 1X \sum_{N=1}^\infty E(p_N) <
  \frac 1X \sum_{N=1}^\infty C^2 e^{-\delta N/3}.
  $$
  The last series converges, therefore $\mu(\{g:\sum_{N=1}^\infty p_N(g) =\infty\})=0$.
  The sum $\sum_{N=1}^\infty p_N(g)$ diverges exactly when the family in
  consideration has essential infinitesimal cyclicity at least 2. Thus,
  $\mu(f+E)=0$ and the theorem is proved.
\end{theorem-proof}

\begin{theorem-proof}{thr:cubenull}
  First, consider the case of finitely differentiable functions
  $C^k(S)$.

  In the definition of the cube measure set
  $x_0=\frac{1-e^{-1}\cos(x)}{1+e^{-2}-2e^{-1}\cos(x)}$, $x_n=e^{-n^2}
  \sin(nx/2)$ if $2 | n$ and $x_n=e^{-n^2}\cos((n-1)x/2)$
  otherwise. Let $\mu$ be a cube measure on $C^k(S)$ defined by these
  settings. Notice that the Fourier expansion of $x_0$ is
  $2+\sum_{m=1}^\infty e^{-m} \cos(mx)$.

  Let $f\in \supp(\mu)$ and $\tilde f_n$ be the Fourier coefficients
  of $f$. A direct computation shows that for large values of $N$
  \begin{eqnarray*}
    |\tilde f_N| > \frac{2}{\cosh N-1} \left\|\sum_{m=2}^\infty \tilde
    f_{-Nm} e^{-iNmx}+ \tilde f_{Nm} e^{iNmx} \right\|_{C^\omega_1(S)}.
  \end{eqnarray*}
  Arguing as in the proof of Theorem~\ref{thr:comeganull} and using
  Proposition~\ref{pr:comega} and Lemma~\ref{lm:cos} we see that the
  family corresponding to $f$ has infinitesimal cyclicity essentially
  bounded by one. Thus, $\mu(\ic_\infty(C^k(S)))=0$.

  Now let us prove that the set $\ic_{\mbox{{\small fin}}}(C^\omega_\delta(S))$ is not cube prevalent.

  We leave the choice of $x_n$ as before, but now set
  $x_0=\sum_{m=1}^\infty e^{-m^{m!} \delta} \cos(m^{m!}x)$ and let $\mu$ be
  the corresponding cube measure on $C^\omega_\delta(S)$.
  In this case one can check that for any $f\in \supp(\mu)$ and for
  $N$ sufficiently large the following inequality holds:
  $$
  \frac 1{4^{M-1}}|\tilde f_{NM}| \ge \sum_{m=1}^{N-1} |\tilde f_Mm|,
  $$
  where $M=N^{(N-1)!}$. Using the same reasoning as in the proof of
  Theorem~\ref{thr:ckporous} and using Proposition~\ref{pr:ck} and
  Theorem~\ref{thr:nsbs} we can conclude that the family corresponding
  to $f$ has at least $N$ periodic attractors for a suitable choice of
  the parameters. Since $N$ can be arbitrarily large we can see that
  $\mu(\ic_{\mbox{{\small fin}}}(C^\omega_\delta(S)))=0$. 
\end{theorem-proof}

\section{Perturbations of constant vector fields on the torus}
\label{sec:tourus-case}

Let $\phi^t_{\alpha,a}$ be the flow of the vector field $V_{\alpha,a}$ which we
assume to be at least $C^2$. The
perturbation theory tells us that for small values of $a$ this flow
can be expressed as
$$
\phi^t_{\alpha,a}:  
 \left(
    \begin{array}{c}
      x_1\\
      x_2
    \end{array}
  \right)
\mapsto
 \left(
    \begin{array}{c}
      x_1\\
      x_2
    \end{array}
  \right)
+
  \left(
    \begin{array}{c}
      \cos \alpha\\
      \sin \alpha
    \end{array}
  \right)  t
 + 
 a \int_0^t\left(
    \begin{array}{c}
      v_1(x_1+\cos \alpha \tau,x_2+\sin \alpha \tau)\\
      v_2(x_1+\cos \alpha \tau,x_2+\sin \alpha \tau)
    \end{array}
    \right)
    \, d\tau
 + O(a^2)
$$

Now, assume $\cos \alpha \neq 0$ and compute the Poincare map to the circle
$x_1=0$. This map is
\begin{eqnarray*}
x_2 &\mapsto& x_2+2\pi \tan \alpha
- a \tan \alpha \int_0^{2\pi/\cos \alpha} v_1(\cos \alpha \tau, x_2+ \sin \alpha \tau) \, d\tau\\
&&+ a \int_0^{2\pi/\cos \alpha} v_2(\cos \alpha \tau, x_2+ \sin \alpha \tau) \, d\tau
+O(a^2).
\end{eqnarray*}
Let $\tilde v_{1,m,n}$, $\tilde v_{2,m,n}$ be the Fourier coefficients
of $v_1$ and $v_2$. Define
$$
\tilde f_n(\alpha)=
 \sum_{m=-\infty}^{\infty}
 \left( -\tan \alpha \tilde v_{1,m,n}
   +\tilde v_{2,m,n}\right) \psi_{m,n}(\alpha),
$$
where $\psi_{m,n}(\alpha)$ is given by
$$
\psi_{m,n}(\alpha)=\left\{
  \begin{array}{ll}
    2\pi/\cos \alpha, & \mbox{ if } m\cos \alpha + n \sin \alpha = 0\\
    -i\frac{e^{2\pi i n\tan \alpha}-1}{m\cos \alpha + n \sin \alpha}, &
    \mbox{ otherwise}
  \end{array}
  \right.
$$

Define the function $f(x,\alpha)=\sum_{n=-\infty}^\infty\tilde
f_n(\alpha)e^{inx}$. It is easy to see that the Poincare map computed
above can be written as
$$
x \mapsto x + 2\pi \tan \alpha + a f(x,\alpha) + O(a^2).
$$
This family looks almost as the families $F$ we have been studying, but
here the function $f$ depends on the parameter $\alpha$.
However it does not make any difference for the proof of the analog of
Theorem~\ref{thr:iic}: if $v_1$ and $v_2$ are trigonometric
polynomials, then $f$ is also trigonometric polynomial (with
coefficients depending on the parameter $\alpha$) and all arguments in
the proof of Theorem~\ref{thr:iic} go through.

Another way to compute the Fourier coefficients $\tilde f_n$ is the
following. Make the Fourier transform of $v_1$ and $v_2$ only with
respect to $x_2$ and denote by $\breve v_{1,n}$, $\breve
v_{2,n}$ the Fourier coefficients of $v_1$ and $v_2$ where the Fourier
transform is made only with respect to $x_2$, i.e.
$$
v_j(x_1,x_2)=\sum_{n=-\infty}^\infty \breve v_{j,n}(x_1) e^{inx_2},
$$
where $j=1,2$. Then
$$
\tilde f_n(\alpha)=
\int_0^{2\pi/\cos \alpha} 
\left( -\tan \alpha\, \breve v_{1,n}(\cos(\alpha \tau)) + \breve
  v_{2,n}(\cos(\alpha \tau))\right) e^{in\sin(\alpha \tau)} \, d\tau.
$$

In the proofs in previous sections we use several times
the Riemann-Lebesgue lemma. The analog of this lemma also holds for
$\breve v_{j,n}$. We will formulate this lemma when $v_1$ and $v_2$ depend
on a parameter $t\in I^b$, so this statement can be applied to the
proof of the analog of Theorem~\ref{thr:ck-ic} directly. If $v_1$,
$v_2$ do not depend on a parameter, just set $b=0$.

\begin{lemma}\label{lm:v-uniform}
  For any $\epsilon_0>0$ there exists
  $N=N(\epsilon_0)$ such that for all $t\in I^b$, $x_1\in \R$, $n\in \Z$, $|n|\ge N$ one has
$$
    |{\breve v_{j,n}}^{(k_1,m)}(x_1,t)| < \epsilon_0 |n|^{-k_2},
  $$
  for all $m\le l$, $k_1,k_2 \in \N$ such that $k_1+k_2=k$,  where ${\breve v_{j,n}}^{(k_1,m)}(x_1,t)$ denotes the $k_1$th derivative
  of ${\breve v_{j,n}}(x_1,t)$ with respect to $x_1$ and $m$th
  derivative with respect to $t$.
\end{lemma}
The proof of this lemma is identical to the proof of
Lemma~\ref{lm:uniform}.

Notice, that this lemma implies that for any $\epsilon_0>0$ there is
$N$ such that for all $n\in \Z$, $|n|>N$ one has
$$
|\tilde f_n| < \epsilon_0 |n|^{-k}.
$$

Using these settings the proofs of the analogs of
Theorems~\ref{thr:iic}, \ref{thr:ck-ic}, \ref{thr:cknull} and \ref{thr:ckporous} go
along the same lines as before. For example, 
as the perturbed family in the proof of Theorem~\ref{thr:ck-ic} one should consider
\begin{eqnarray*}
\hat v_j(x_1,x_2,t)&=&v_j(x_1,x_2,t) - \sum_{m=1}^{d+1}
(\breve v_{j,-mN}(x_1,t) e^{-imNx_2}+\breve v_{j,mN}(x_1,t) e^{imNx_2}) \\
&&+ \epsilon N^{-k}(d+1)^{-k}\sin(N(d+1)x_2).
\end{eqnarray*}

In the case of Theorem~\ref{thr:comeganull} the measure $\mu$ is
defined as the image of the measure $\tau$ under the map $\phi :
Q^{\aleph_0} \to C^\omega_\delta (T^2) \times  C^\omega_\delta (T^2)$
defined by
$$
\phi((w_1,w_2,\ldots))=\left(0,\sum_{n=1}^\infty e^{-3/2\delta
    n}(w_{n}e^{inx_2}+\bar w_n e^{-inx_2})\right).
$$
Again, the rest of the proof can be easily adjusted.

\bibliographystyle{alpha}
\bibliography{circle_infi}

\end{document}